\documentclass{article}
\usepackage{graphics}
\usepackage{graphicx}
\usepackage{latexsym}
\usepackage{amssymb}
\usepackage{amsmath}
\usepackage[normalem]{ulem}
\newtheorem{theorem}{Theorem}[section]
\newtheorem{lemma}[theorem]{Lemma}
\newtheorem{proposition}[theorem]{Proposition}
\newtheorem{corollary}[theorem]{Corollary}
\newtheorem{definition}{Definition}[section]
\newtheorem{preexample}{Example}[section]
\newenvironment{example}{\begin{preexample}}{\end{preexample}}
\newtheorem{preremark}{Remark}
\newenvironment{remark}{\begin{preremark}\rm}{\end{preremark}}

\newenvironment{proof}
  {{\bf Proof:}}
  {\qquad \hspace*{\fill} $\Box$}%

\begin{document}

\title{Algebraic Conjugacy of Linear Flows on Connected Lie Groups}
\author{S. N. Stelmastchuk\\ Universidade Federal do Paran\'{a}\\Jandaia do Sul, Brazil, simnaos@ufpr.br}

\maketitle

\begin{abstract}
  We study algebraic conjugacy of linear flows on connected Lie groups. To each linear flow we associate a derivation of the Lie algebra, and we analyze how the conjugacy problem can be transferred to the infinitesimal level. In the connected and simply connected case, algebraic conjugacy is characterized by the conjugacy of the associated derivations. For non-simply connected Lie groups, we show that this condition must be supplemented by global restrictions arising from the universal covering group. More precisely, only projectable derivations define linear flows on the quotient, and only admissible Lie algebra isomorphisms descend to conjugacies between the quotient groups. This provides a framework for the algebraic classification of linear flows by computing projectable derivations and admissible automorphisms. We also describe fixed point subgroups and kernels of derivations as obstructions to algebraic conjugacy. Finally, we apply the general results to the three-dimensional Heisenberg group and to its central quotients, showing explicitly how the quotient topology restricts both the admissible derivations and the conjugating automorphisms.
\end{abstract}



\section{Introduction}

The notion of conjugacy is a central tool in the qualitative theory of dynamical systems. Roughly speaking, two dynamical systems are conjugate when there exists a change of variables that transforms the trajectories of one system into the
trajectories of the other, preserving the relevant dynamical structure. In this sense, conjugacy allows one to replace a given system by a simpler one without losing its essential topological, differentiable or algebraic properties. Classical
references present several aspects of this problem; see, for instance, \cite{chicone,colonius,Robi99}.

In this paper, the ambient space is a connected Lie group $G$, and the dynamical systems under consideration are linear flows on $G$. A linear flow is the flow of a linear vector field $\mathcal X$, that is, a vector field whose flow $(\varphi_t)_{t\in\mathbb R}$ is a one-parameter group of automorphisms of $G$. Equivalently, the trajectories are solutions of
\[
  \dot g=\mathcal X(g), \qquad g\in G.
\]
A fundamental fact is that every linear vector field has an associated derivation $D\in\operatorname{Der}(\mathfrak g)$, where $\mathfrak g$ is the Lie algebra of $G$, and the differential of the flow at the identity satisfies $(d\varphi_t)_e=e^{tD}$.

A first study of conjugacy for linear systems on Lie groups was developed in \cite{DaSan1}, where the authors treated a particular
nilpotent setting. Although topological conjugacy is a natural and important notion in dynamical systems, in the present paper we restrict ourselves to a stronger class of conjugacies. We say that two linear flows are algebraically conjugate when they are conjugate by a Lie group isomorphism. This restriction is natural in Lie theory, since a Lie group isomorphism induces a Lie algebra
isomorphism( see for instance \cite{Sanmartin}). Therefore, one may try to study conjugacy of linear flows through the corresponding conjugacy problem for derivations on Lie algebras.

At first sight, this reduction to the Lie algebra may seem straightforward. However, a difficulty appears when the Lie group is connected but not simply connected. If $G$ is connected and simply connected, every derivation of $\mathfrak g$ integrates to a unique linear flow on $G$. For a general connected Lie group, one must pass to the universal covering group $\widetilde G$. A derivation $D\in\operatorname{Der}(\mathfrak g)$ always generates a linear flow $(\widetilde\varphi_t^D)_{t\in\mathbb R}$ on
$\widetilde G$, but this lifted flow does not necessarily descend to $G$. If $G\simeq \widetilde G/\Gamma_G$, then the descent condition is
\[
  \Gamma_G\subset \operatorname{Fix}(\widetilde\varphi^D).
\]
This leads to the class $\operatorname{Der}(\mathfrak g;\Gamma_G)$ of projectable derivations.

After identifying which derivations generate linear flows on the quotient, one must also determine which Lie algebra isomorphisms are compatible with the quotient structure. Suppose that $H\simeq \widetilde H/\Gamma_H$ is another connected Lie group. A Lie algebra isomorphism $A:\mathfrak g\to\mathfrak h$ integrates to an isomorphism $\widetilde A:\widetilde G\longrightarrow \widetilde H$ between the universal covering groups. In order for this isomorphism to descend to an isomorphism from $G$ onto $H$, it is necessary and sufficient that
\[
  \widetilde A(\Gamma_G)=\Gamma_H.
\]
This motivates the class $\operatorname{Iso}(\mathfrak g,\mathfrak h;\Gamma_G,\Gamma_H)$ of admissible Lie algebra isomorphisms.

The main result of this paper combines these two restrictions. We show that two linear flows on connected Lie groups are algebraically conjugate if and only if their associated derivations are conjugate by an admissible Lie algebra isomorphism. More precisely, the conjugacy criterion involves derivations $D\in\operatorname{Der}(\mathfrak g;\Gamma_G)$, $F\in\operatorname{Der}(\mathfrak h;\Gamma_H)$,
and an isomorphism $A\in\operatorname{Iso}(\mathfrak g,\mathfrak h;\Gamma_G,\Gamma_H)$ satisfying
\[
  AD=FA.
\]
Thus the algebraic classification is governed not by the full action of $\operatorname{Aut}(\mathfrak g)$ on $\operatorname{Der}(\mathfrak g)$, but by a restricted action determined by the topology of the quotient. This point of view suggests a classification program: for each connected Lie group $G=\widetilde G/\Gamma_G$, one has to compute the projectable derivations, the admissible automorphisms, and the corresponding restricted orbits.

We also discuss obstructions to algebraic conjugacy. The first obstruction comes from fixed point subgroups: if two linear flows are algebraically conjugate, then their fixed point subgroups must be isomorphic as Lie groups. The second one is infinitesimal: if $D$ and $F$ are the associated derivations, then any algebraic conjugacy induces an isomorphism satisfying
\[
  A(\ker D)=\ker F.
\]
Consequently, if $\ker D$ and $\ker F$ are not isomorphic as Lie algebras, then the corresponding linear flows cannot be algebraically conjugate.

Beyond the general theory, we study linear flows on the three-dimensional Heisenberg group. In the connected and simply connected case, every derivation of $\mathfrak h_3$ generates a linear flow, and the algebraic classification is described through the action of $\operatorname{Aut}(\mathfrak h_3)$ on $\operatorname{Der}(\mathfrak h_3)$. We then consider the non-simply connected
central quotients
\[
  H_{3,a}=\widetilde H_3/\Gamma_a,
  \qquad
  \Gamma_a=\{\exp(naZ);\ n\in\mathbb Z\}.
\]
In this quotient case, projectability imposes the condition
\[
  \operatorname{tr}(A)=0,
\]
and admissibility restricts the automorphisms to those whose horizontal block satisfies $\det P=\pm1$. This shows that the algebraic classification on $H_{3,a}$ is genuinely different from the simply connected case. Thus the Heisenberg case works as a model example for applying the general criterion to other classes of Lie groups.

The paper is organized as follows. In Section~2 we recall the basic facts on linear vector fields, linear flows and derivations. Section~3 establishes the criterion for algebraic conjugacy of linear flows on connected Lie groups, with emphasis on the role of universal coverings and admissible isomorphisms. In Section~4 we study obstructions to algebraic conjugacy arising from fixed point
subgroups and kernels of derivations. Finally, Section~5 applies the previous results to the Heisenberg group and to its central quotients.

\section{Linear Flows and Derivations}

In this section we recall the basic notions concerning linear vector fields, linear flows and their associated derivations. Our aim is to fix the notation and to make explicit the link between the dynamics on a Lie group and the induced linear dynamics on its Lie algebra. We follow the standard terminology used in the theory of linear control systems on Lie groups; see, for instance,
\cite{jouanI,jouanII}.

Let $G$ be a connected Lie group with identity element $e$ and Lie algebra $\mathfrak g=T_eG$. A smooth vector field $\mathcal X$ on $G$ determines, whenever it is complete, a global flow
\[
  (\varphi_t)_{t\in\mathbb R},
  \qquad
  \varphi_t:G\longrightarrow G,
\]
where $\varphi_t(g)$ is the solution at time $t$ of the differential equation defined by $\mathcal X$ with initial condition $g$. Thus
\[
  \varphi_0=\operatorname{id}_G, \qquad \varphi_{t+s}=\varphi_t\circ\varphi_s,
\]
for all $t,s\in\mathbb R$.

The flow of an arbitrary vector field is a one-parameter group of diffeomorphisms. A linear vector field is a vector field whose flow is more rigid: each time map is required to be a Lie group automorphism.

\begin{definition}
  A complete vector field $\mathcal X$ on a connected Lie group $G$ is called \emph{linear} if its flow $(\varphi_t)_{t\in\mathbb R}$ satisfies $\varphi_t\in \operatorname{Aut}(G)$ for every $t\in\mathbb R$. Equivalently,
  \[
    \varphi_t(gh)=\varphi_t(g)\varphi_t(h), \qquad g,h\in G,\ t\in\mathbb R.
  \]
  In this case, $(\varphi_t)_{t\in\mathbb R}$ is called a \emph{linear flow} on $G$.
\end{definition}
Since every Lie group automorphism fixes the identity element, a linear flow  satisfies $\varphi_t(e)=e$ for every $t\in\mathbb R$. Hence the corresponding linear vector field satisfies $\mathcal X(e)=0$. Thus linear vector fields should not be confused with left-invariant or right-invariant vector fields. In general, a nonzero invariant vector field does  not vanish at the identity.

This definition should be viewed as the natural analogue, on a Lie group, of a  linear differential equation on a vector space. Indeed, if $G=\mathbb R^n$  with its additive group structure, then $\operatorname{Aut}(G)=GL(n,\mathbb R)$. A linear flow on $G$ is therefore a one-parameter subgroup of $GL(n,\mathbb R)$, and hence has the form
\[
  \varphi_t(x)=e^{tA}x, \qquad x\in\mathbb R^n,
\]
for some matrix $A\in\mathfrak{gl}(n,\mathbb R)$. Thus the usual linear system
\[
  \dot x=Ax
\]
is recovered as a particular case.

For a general Lie group, the corresponding infinitesimal object is not a  matrix, but a derivation of the Lie algebra.

\begin{definition}
  A linear map $D:\mathfrak g\longrightarrow\mathfrak g$ is called a  \emph{derivation} of $\mathfrak g$ if
  \[
    D[X,Y]=[DX,Y]+[X,DY],\qquad X,Y\in\mathfrak g.
  \]
  The vector space of all derivations of $\mathfrak g$ is denoted by  $\operatorname{Der}(\mathfrak g)$.
\end{definition}

Let $\mathcal X$ be a linear vector field on $G$, with flow $(\varphi_t)_{t\in\mathbb R}$. Since each $\varphi_t$ is a Lie group
automorphism, its differential at the identity,
\[
  (d\varphi_t)_e:\mathfrak g\longrightarrow\mathfrak g,
\]
is a Lie algebra automorphism. Therefore
\[
  t\longmapsto (d\varphi_t)_e
\]
is a one-parameter group of Lie algebra automorphisms of $\mathfrak g$. Its infinitesimal generator is a derivation $D\in\operatorname{Der}(\mathfrak g)$,  and it is characterized by
\[
  (d\varphi_t)_e=e^{tD}, \qquad t\in\mathbb R
\]
We call $D$ the derivation associated with the linear flow  $(\varphi_t)_{t\in\mathbb R}$.

Equivalently, if $Y$ is a right-invariant vector field on $G$, then the associated derivation is given by
\[
  DY=-[\mathcal X,Y],
\]
where the bracket is the usual Lie bracket of vector fields. Through the natural  identification between right-invariant vector fields and elements of  $\mathfrak g$, this formula defines an element of $\operatorname{Der}(\mathfrak g)$.

The equality
\[
  (d\varphi_t)_e=e^{tD}
\]
shows that the linear flow $(\varphi_t)_{t\in\mathbb R}$ induces a linear flow on the tangent space at the identity, that is, on the Lie algebra $\mathfrak g=T_eG$. Indeed, given $X_0\in\mathfrak g$, define
\[
  X(t):=(d\varphi_t)_eX_0.
\]
Then $X(t)=e^{tD}X_0$. Consequently, $\dot X(t)=DX(t)$, and $X(t)$ is the solution of the linear differential equation
\[
  \dot X=DX,\qquad X(0)=X_0.
\]
Thus the derivation $D$ describes the infinitesimal dynamics induced by the linear flow on the Lie algebra $\mathfrak g$. This does not mean that the global dynamics on $G$ is completely determined without further hypotheses; it means that the differential of the flow at the identity is governed by the linear system generated by $D$.

There is, however, an important distinction between the simply connected and  the non-simply connected cases. If $G$ is connected and simply connected, then  every derivation $D\in\operatorname{Der}(\mathfrak g)$ integrates to a unique one-parameter group of automorphisms of $G$. Hence, in this case, every derivation defines a unique linear flow on $G$.

For a connected Lie group which is not simply connected, this statement is no longer automatic. A derivation always defines a linear flow on the universal  covering group $\widetilde G$, but this lifted flow may fail to descend to the
quotient group $G$. Thus, in the non-simply connected case, one must impose an  additional compatibility condition with the discrete subgroup defining the  quotient. This global obstruction will be discussed in the next section.

We shall use the following notation throughout the paper. If  $(\varphi_t)_{t\in\mathbb R}$ is a linear flow on $G$, we denote by
$D\in\operatorname{Der}(\mathfrak g)$ its associated derivation. Conversely,  when the Lie group under consideration is connected and simply connected, or  when the relevant projectability condition is satisfied, we denote by
\[
  (\varphi_t^D)_{t\in\mathbb R}
\]
the linear flow generated by the derivation $D$.

%---------------
%---------------

\section{Algebraic Conjugacy and Descent from Universal Coverings}
\label{sec:algebraic_conjugacy_descent}

In this section we study algebraic conjugacy of linear flows on connected Lie groups. The main point is that the infinitesimal conjugacy condition on the Lie algebras gives, in the simply connected setting, a conjugacy between the lifted flows on the universal covering groups. However, this condition does not automatically give a conjugacy on the original Lie groups. In order to descend the conjugacy, one must impose a compatibility condition with the discrete central subgroups defining the quotients.

Throughout this section, $G$ and $H$ denote connected Lie groups with Lie algebras $\mathfrak g$ and $\mathfrak h$, respectively. Let $p_G:\widetilde G\longrightarrow G$, $p_H:\widetilde H\longrightarrow H$ be their universal covering homomorphisms. We denote $\Gamma_G=\ker p_G$, $\Gamma_H=\ker p_H$. Then $\Gamma_G\subset Z(\widetilde G)$ and $\Gamma_H\subset Z(\widetilde H)$ are discrete central subgroups, and
\[
  G\simeq \widetilde G/\Gamma_G,\qquad  H\simeq \widetilde H/\Gamma_H.
\]
Let $(\varphi_t)_{t\in\mathbb R}$ and $(\psi_t)_{t\in\mathbb R}$ be linear flows on $G$ and $H$, respectively, and $D\in\operatorname{Der}(\mathfrak g)$,  $F\in\operatorname{Der}(\mathfrak h)$ be their associated derivations.

In this paper we restrict ourselves to conjugacies given by Lie group isomorphisms. This is stronger than topological conjugacy and will be called algebraic conjugacy.

\begin{definition}
  The linear flows $(\varphi_t)_{t\in\mathbb R}$ and  $(\psi_t)_{t\in\mathbb R}$ are said to be algebraically conjugate if there  exists a Lie group isomorphism $\Phi:G\longrightarrow H$ such that
  \[
    \Phi\circ\varphi_t=\psi_t\circ\Phi
  \]
  for every $t\in\mathbb R$.
\end{definition}

We first recall the simply connected case. In this situation there is no
topological obstruction to the passage from Lie algebras to Lie groups.

\begin{proposition} \label{prop:sc_conjugacy_derivations}
  Let $G$ and $H$ be connected and simply connected Lie groups with Lie algebras $\mathfrak g$ and $\mathfrak h$, respectively. Let $(\varphi_t)_{t\in\mathbb R}$ and $(\psi_t)_{t\in\mathbb R}$ be linear flows on $G$ and $H$, with associated  derivations $D\in\operatorname{Der}(\mathfrak g)$, $F\in\operatorname{Der}(\mathfrak h)$. Then $(\varphi_t)_{t\in\mathbb R}$ and $(\psi_t)_{t\in\mathbb R}$ are algebraically conjugate if and only if there exists a Lie algebra isomorphism $A:\mathfrak g\longrightarrow\mathfrak h$ such that
  \[
    AD=FA.
  \]
\end{proposition}
\begin{proof}
  Suppose first that the flows are algebraically conjugate by a Lie group isomorphism $\Phi:G\longrightarrow H$. Taking differentials at the identity in the relation $\Phi\circ\varphi_t=\psi_t\circ\Phi$ we obtain
  \[
    d\Phi_e\circ (d\varphi_t)_e=(d\psi_t)_e\circ d\Phi_e.
  \]
  Since $(d\varphi_t)_e=e^{tD}$,  $(d\psi_t)_e=e^{tF}$, and writing $A=d\Phi_e$, we get $Ae^{tD}=e^{tF}A$ for every $t\in\mathbb R$. Differentiating at $t=0$, we obtain
  \[
    AD=FA.
  \]
  Conversely, suppose that there exists a Lie algebra isomorphism $A:\mathfrak g\longrightarrow\mathfrak h$ such that $AD=FA$. Since $G$ and $H$ are connected and simply connected, there exists a unique  Lie group isomorphism $\Phi:G\longrightarrow H$ such that $d\Phi_e=A$. The relation $AD=FA$ implies
  \[
    Ae^{tD}=e^{tF}A
  \]
  for every $t\in\mathbb R$. Therefore,
  \[
    d(\Phi\circ\varphi_t)_e=d(\psi_t\circ\Phi)_e.
  \]
  Since Lie group homomorphisms defined on a connected Lie group are determined by their differentials at the identity, it follows that
  \[
    \Phi\circ\varphi_t=\psi_t\circ\Phi
  \]
  for every $t\in\mathbb R$( see for instance \cite{hall}). Thus the flows are algebraically conjugate.
\end{proof}

The previous proposition shows that, on universal covering groups, algebraic conjugacy is completely described by conjugacy of derivations. We now describe which derivations define linear flows on a quotient Lie group.

Let $G=\widetilde G/\Gamma_G$, where $\widetilde G$ is connected and simply connected and $\Gamma_G\subset Z(\widetilde G)$ is a discrete central subgroup. For $D\in\operatorname{Der}(\mathfrak g)$, denote by $(\widetilde\varphi_t^D)_{t\in\mathbb R}$ the linear flow on $\widetilde G$ associated with $D$, it exists by Theorem 2 in \cite{jouanII}.

\begin{proposition} \label{prop:projectable_derivations}
  The flow $(\widetilde\varphi_t^D)_{t\in\mathbb R}$ projects to a linear flow on $G=\widetilde G/\Gamma_G$ if and only if
  \[
    \Gamma_G\subset \operatorname{Fix}(\widetilde\varphi^D),
  \]
  that is, $\widetilde\varphi_t^D(\gamma)=\gamma$ for every $\gamma\in\Gamma_G$ and every $t\in\mathbb R$.
\end{proposition}
\begin{proof}
  Suppose first that $(\widetilde\varphi_t^D)_{t\in\mathbb R}$ projects to a  linear flow $(\varphi_t)_{t\in\mathbb R}$ on $G$. Then
  \[
    p_G\circ\widetilde\varphi_t^D=\varphi_t\circ p_G
  \]
  for every $t\in\mathbb R$. If $\gamma\in\Gamma_G$, then
  \[
    p_G(\widetilde\varphi_t^D(\gamma))  = \varphi_t(p_G(\gamma))  = \varphi_t(e)  = e.
  \]
  Hence $\widetilde\varphi_t^D(\gamma)\in\Gamma_G$. Since $\Gamma_G$ is discrete and the map $t\longmapsto \widetilde\varphi_t^D(\gamma)$ is continuous, this map is constant. As it is equal to $\gamma$ at $t=0$, we obtain $\widetilde\varphi_t^D(\gamma)=\gamma$ for every $t\in\mathbb R$. Thus 
  \[
    \Gamma_G\subset \operatorname{Fix}(\widetilde\varphi^D).
  \]

  Conversely, suppose that $\Gamma_G\subset \operatorname{Fix}(\widetilde\varphi^D)$. Define $\varphi_t:G\longrightarrow G$ by
  \[
    \varphi_t(p_G(\widetilde g))  = p_G(\widetilde\varphi_t^D(\widetilde g)).
  \]
  We show that $\varphi_t$ is well defined. If $p_G(\widetilde g_1)=p_G(\widetilde g_2)$, then there exists $\gamma\in\Gamma_G$ such that $\widetilde g_2=\widetilde g_1\gamma$. Since $\widetilde\varphi_t^D$ is an automorphism of $\widetilde G$, we have
  \[
    \widetilde\varphi_t^D(\widetilde g_2) = \widetilde\varphi_t^D(\widetilde g_1\gamma) = \widetilde\varphi_t^D(\widetilde g_1)\widetilde\varphi_t^D(\gamma).
  \]
  Since $\gamma\in\operatorname{Fix}(\widetilde\varphi^D)$, this gives $\widetilde\varphi_t^D(\widetilde g_2) = \widetilde\varphi_t^D(\widetilde g_1)\gamma$. Therefore,
  \[
    p_G(\widetilde\varphi_t^D(\widetilde g_2))  = p_G(\widetilde\varphi_t^D(\widetilde g_1)).
  \]
  Thus $\varphi_t$ is well defined. Since each $\widetilde\varphi_t^D$ is a Lie group automorphism, each  $\varphi_t$ is also a Lie group automorphism. Moreover,
  \[
    \varphi_{t+s}=\varphi_t\circ\varphi_s
  \]
  follows from the corresponding identity on $\widetilde G$. Hence  $(\varphi_t)_{t\in\mathbb R}$ is a linear flow on $G$.
\end{proof}

This motivates the following definition.

\begin{definition}
Let $G=\widetilde G/\Gamma_G$ be a connected Lie group. We define
\[
\operatorname{Der}(\mathfrak g;\Gamma_G)
=
\{D\in\operatorname{Der}(\mathfrak g);
\Gamma_G\subset\operatorname{Fix}(\widetilde\varphi^D)\},
\]
where $(\widetilde\varphi_t^D)_{t\in\mathbb R}$ is the linear flow on
$\widetilde G$ associated with $D$.
\end{definition}

Thus $\operatorname{Der}(\mathfrak g;\Gamma_G)$ is the set of derivations whose associated linear flows on the universal covering group descend to linear flows on $G$. The following elementary example shows that projectability is a genuine restriction and not a formal condition.

\begin{example}\label{ex:projectable_derivations_cylinder}
  Let $\widetilde G=\mathbb R^2$ and $G=\mathbb R\times S^1$. We identify $G$ with the quotient $G=\mathbb R^2/\Gamma_G$, where $\Gamma_G=\{(0,n);\, n\in\mathbb Z\}$. Thus $\Gamma_G$ is a discrete central subgroup of the universal covering group $\widetilde G=\mathbb R^2$.

  Since $\mathbb R^2$ is abelian, every linear map is a derivation. Hence $\operatorname{Der}(\mathbb R^2)=\mathfrak{gl}(2,\mathbb R)$. Let
  \[
    D=
    \begin{pmatrix}
      a & b\\
      c & d
    \end{pmatrix}
    \in\mathfrak{gl}(2,\mathbb R).
  \]
  The linear flow on $\widetilde G=\mathbb R^2$ associated with $D$ is $\widetilde\varphi_t^D(x)=e^{tD}x$. We want to determine when this flow descends to a linear flow on $G=\mathbb R^2/\Gamma_G$. By the projection criterion, this happens if and only if $\Gamma_G\subset \operatorname{Fix}(\widetilde\varphi^D)$. Since $\Gamma_G$ is generated by the element $(0,1)$, this condition is equivalent to $ \widetilde\varphi_t^D(0,1)=(0,1)$ for every $t\in\mathbb R$. In other words, $e^{tD}(0,1)=(0,1)$ for every $t\in\mathbb R$. Differentiating this identity at $t=0$, we obtain $D(0,1)=0$. Since $D(0,1)=(b,d)$, the condition becomes
  \[
    b=0 \qquad\text{and}\qquad  d=0.
  \]
  Consequently,
  \[
    \operatorname{Der}(\mathbb R^2;\Gamma_G) = 
    \left\{
    \begin{pmatrix}
      a & 0\\
      c & 0
    \end{pmatrix}; \ a,c\in\mathbb R
    \right\}.
  \]
  This example shows that the set of projectable derivations may be a proper  nontrivial subset of the full derivation algebra. Indeed,
  \[
    \operatorname{Der}(\mathbb R^2;\Gamma_G) \subsetneq \operatorname{Der}(\mathbb R^2)=\mathfrak{gl}(2,\mathbb R),
  \]
  but
  \[
    \operatorname{Der}(\mathbb R^2;\Gamma_G)\neq\{0\}.
  \]
  Thus there are derivations which generate linear flows on the universal covering  group $\mathbb R^2$, but do not generate linear flows on the quotient group $\mathbb R\times S^1$.
\end{example}

We now introduce the corresponding notion for Lie algebra isomorphisms.

\begin{definition}
  Let $G=\widetilde G/\Gamma_G$, $H=\widetilde H/\Gamma_H$ be connected Lie groups. We define
  \[
    \operatorname{Iso}(\mathfrak g,\mathfrak h;\Gamma_G,\Gamma_H)
  \]
  as the set of all Lie algebra isomorphisms $A:\mathfrak g\longrightarrow\mathfrak h$ such that, if $\widetilde A:\widetilde G\longrightarrow\widetilde H$ is the unique Lie group isomorphism satisfying $d\widetilde A_{\widetilde e}=A$, then $\widetilde A(\Gamma_G)=\Gamma_H$. The elements of $\operatorname{Iso}(\mathfrak g,\mathfrak h;\Gamma_G,\Gamma_H)$ will be called $(\Gamma_G,\Gamma_H)$-admissible isomorphisms.
\end{definition}

In the particular case $G=H$, we write
\[
  \operatorname{Aut}(\mathfrak g;\Gamma_G)  = \operatorname{Iso}(\mathfrak g,\mathfrak g;\Gamma_G,\Gamma_G) =
  \{A\in\operatorname{Aut}(\mathfrak g);  \widetilde A(\Gamma_G)=\Gamma_G\}.
\]

\begin{lemma} \label{lem:conjugacy_lifts}
  Let $p_G:\widetilde G\longrightarrow G$, $p_H:\widetilde H\longrightarrow H$ be the universal covering homomorphisms of the connected Lie groups $G$ and $H$. Let $(\varphi_t)_{t\in\mathbb R}$ and $(\psi_t)_{t\in\mathbb R}$ be linear flows on $G$ and $H$, respectively, and let $(\widetilde\varphi_t)_{t\in\mathbb R}$ and $(\widetilde\psi_t)_{t\in\mathbb R}$be their lifts to $\widetilde G$ and $\widetilde H$ fixing the identity. Suppose that $\Phi:G\longrightarrow H$ is a Lie group isomorphism such that
  \[
    \Phi\circ\varphi_t=\psi_t\circ\Phi
  \]
  for every $t\in\mathbb R$. Let  $\widetilde\Phi:\widetilde G\longrightarrow\widetilde H$ be the unique lift of $\Phi$ satisfying $
    p_H\circ\widetilde\Phi=\Phi\circ p_G$ and $\widetilde\Phi(\widetilde e)=\widetilde e$. Then
  \[
    \widetilde\Phi\circ\widetilde\varphi_t  = \widetilde\psi_t\circ\widetilde\Phi
  \]
  for every $t\in\mathbb R$.
\end{lemma}
\begin{proof}
  Fix $t\in\mathbb R$. Consider the two maps
  \[
    F_1=\widetilde\Phi\circ\widetilde\varphi_t: \widetilde G\longrightarrow\widetilde H \quad \mbox{and} \quad  F_2=\widetilde\psi_t\circ\widetilde\Phi:  \widetilde G\longrightarrow\widetilde H.
  \]
  We will show that $F_1$ and $F_2$ are lifts of the same map from  $\widetilde G$ to $\widetilde H$. Using
  \[
    p_H\circ\widetilde\Phi=\Phi\circ p_G,\,\, p_G\circ\widetilde\varphi_t=\varphi_t\circ p_G\,\, \mbox{and}\,\, \Phi\circ\varphi_t=\psi_t\circ\Phi
  \]
  we get
  \[
    p_H\circ F_1  = \psi_t\circ\Phi\circ p_G = p_H\circ F_2 
  \]
  Thus $F_1$ and $F_2$ are lifts of the same map $\psi_t\circ\Phi\circ p_G:\widetilde G\longrightarrow H$. Moreover, both maps send the identity of $\widetilde G$ to the identity of  $\widetilde H$. Indeed, since the lifted flows fix the identity and $\widetilde\Phi(\widetilde e)=\widetilde e$, we have
  \[
    F_1(\widetilde e) = \widetilde\Phi(\widetilde\varphi_t(\widetilde e)) = \widetilde\Phi(\widetilde e)  = \widetilde e, \quad \mbox{  and} \quad F_2(\widetilde e) = \widetilde\psi_t(\widetilde\Phi(\widetilde e))  = \widetilde\psi_t(\widetilde e)  = \widetilde e.
  \]
  Since $\widetilde G$ is connected, the uniqueness of lifts with prescribed  initial point implies $F_1=F_2$. Therefore,
  \[
    \widetilde\Phi\circ\widetilde\varphi_t  = \widetilde\psi_t\circ\widetilde\Phi.
  \]
  Since $t\in\mathbb R$ was arbitrary, the equality holds for every $t\in\mathbb R$.
\end{proof}

The next theorem is the main criterion for descending conjugacies from universal covering groups to the original Lie groups.

\begin{theorem}\label{thm:general_conjugacy_descent}
  Let $G=\widetilde G/\Gamma_G$, $H=\widetilde H/\Gamma_H$ be connected Lie groups, where $\widetilde G$ and $\widetilde H$ are connected and simply connected. Let $D\in\operatorname{Der}(\mathfrak g;\Gamma_G)$, $F\in\operatorname{Der}(\mathfrak h;\Gamma_H)$. Denote by
  \[
    (\varphi_t^D)_{t\in\mathbb R} \qquad\text{and}\qquad  (\psi_t^F)_{t\in\mathbb R}
  \]
  the projected linear flows on $G$ and $H$, respectively.  Then $(\varphi_t^D)_{t\in\mathbb R}$ and  $(\psi_t^F)_{t\in\mathbb R}$ are algebraically conjugate if and only if there exists
  \[
    A\in\operatorname{Iso}(\mathfrak g,\mathfrak h;\Gamma_G,\Gamma_H) \quad \mbox{such that} \quad AD=FA.
  \]
\end{theorem}
\begin{proof}
  Suppose first that the flows  $(\varphi_t^D)_{t\in\mathbb R}$ and $(\psi_t^F)_{t\in\mathbb R}$ are algebraically conjugate. Then there exists a Lie group isomorphism $\Phi:G\longrightarrow H$ such that $\Phi\circ\varphi_t^D=\psi_t^F\circ\Phi$ for every $t\in\mathbb R$. Let $\widetilde\Phi:\widetilde G\longrightarrow\widetilde H$ be the unique lift of $\Phi$ satisfying $\widetilde\Phi(\widetilde e)=\widetilde e$. Then $\widetilde\Phi$ is a Lie group isomorphism and $p_H\circ\widetilde\Phi=\Phi\circ p_G$. Moreover, $\widetilde\Phi$ maps the kernel of $p_G$ onto the kernel of  $p_H$. Hence
  \[
    \widetilde\Phi(\Gamma_G)=\Gamma_H.
  \]
  Set $A=d\widetilde\Phi_{\widetilde e}$. Then $A:\mathfrak g\to\mathfrak h$ is a Lie algebra isomorphism and $A\in\operatorname{Iso}(\mathfrak g,\mathfrak h;\Gamma_G,\Gamma_H)$. We now prove that $AD=FA$. From Lemma \ref{lem:conjugacy_lifts} we have that the conjugacy relation on $G$ lifts to
  \[
    \widetilde\Phi\circ\widetilde\varphi_t^D  = \widetilde\psi_t^F\circ\widetilde\Phi
  \]
  for every $t\in\mathbb R$. Taking differentials at the identity gives $Ae^{tD}=e^{tF}A$. Differentiating at $t=0$, we obtain $AD=FA$.

  Conversely, suppose that there exists $A\in\operatorname{Iso}(\mathfrak g,\mathfrak h;\Gamma_G,\Gamma_H)$ such that $AD=FA$. Let $    \widetilde A:\widetilde G\longrightarrow\widetilde H$ be the unique Lie group isomorphism satisfying $d\widetilde A_{\widetilde e}=A$. Since $\widetilde G$ and $\widetilde H$ are connected and simply connected,  the derivations $D$ and $F$ determine unique linear flows $(\widetilde\varphi_t^D)_{t\in\mathbb R}$ and $(\widetilde\psi_t^F)_{t\in\mathbb R}$ on $\widetilde G$ and $\widetilde H$, respectively. A direct argument shows that 
  \[
    \widetilde A\circ\widetilde\varphi_t^D  = \widetilde\psi_t^F\circ\widetilde A
  \]
  for every $t\in\mathbb R$. Being $A\in\operatorname{Iso}(\mathfrak g,\mathfrak h;\Gamma_G,\Gamma_H)$, we have $\widetilde A(\Gamma_G)=\Gamma_H$. Therefore, $\widetilde A$ descends to a Lie group isomorphism $\Phi:G\longrightarrow H$ defined by
  \[
    \Phi(p_G(\widetilde g)) = p_H(\widetilde A(\widetilde g)).
  \]
  Finally, using
  \[
    p_G\circ\widetilde\varphi_t^D = \varphi_t^D\circ p_G,\qquad p_H\circ\widetilde\psi_t^F  = \psi_t^F\circ p_H,
  \]
  and
  \[
    \widetilde A\circ\widetilde\varphi_t^D  = \widetilde\psi_t^F\circ\widetilde A,
  \]
  we obtain
  \[
    \Phi\circ\varphi_t^D  = \psi_t^F\circ\Phi
  \]
  for every $t\in\mathbb R$. Thus the projected flows are algebraically conjugate.
\end{proof}

Theorem \ref{thm:general_conjugacy_descent} separates the conjugacy problem into two distinct levels. The first level is infinitesimal and is governed by the equation
\[
  AD=FA.
\]
Together with the integration of $A$ to the universal covering groups, this condition guarantees a conjugacy between the lifted linear flows. The second level is global and is governed by the  admissibility condition
\[
  \widetilde A(\Gamma_G)=\Gamma_H.
\]
This condition is necessary and sufficient for the lifted conjugacy to descend  to the original Lie groups.

Thus, even when two derivations are conjugate at the Lie algebra level, the corresponding linear flows may fail to be conjugate on the original Lie groups  if the conjugating isomorphism of the universal covering groups does not map  $\Gamma_G$ onto $\Gamma_H$.

As an immediate consequence, for a fixed connected Lie group $G=\widetilde G/\Gamma_G$,the algebraic classification of linear flows on $G$ is not governed by the full action of $\operatorname{Aut}(\mathfrak g)$ on
$\operatorname{Der}(\mathfrak g)$, but by the restricted action
\[
  \operatorname{Aut}(\mathfrak g;\Gamma_G)  \times  \operatorname{Der}(\mathfrak g;\Gamma_G)  \longrightarrow \operatorname{Der}(\mathfrak g;\Gamma_G),
\]
given by $(A,D)\longmapsto ADA^{-1}$.

\begin{corollary} \label{cor:classification_fixed_group}
  Let $G=\widetilde G/\Gamma_G$ be a connected Lie group. Let $D,F\in\operatorname{Der}(\mathfrak g;\Gamma_G)$. Then the linear flows on $G$ associated with $D$ and $F$ are algebraically conjugate by an automorphism of $G$ if and only if there exists $A\in\operatorname{Aut}(\mathfrak g;\Gamma_G)$ such that $AD=FA$.
\end{corollary}

\begin{remark}
  If $G$ is connected and simply connected, then $\Gamma_G=\{\widetilde e\}$. Hence
  \[
    \operatorname{Der}(\mathfrak g;\Gamma_G)  = \operatorname{Der}(\mathfrak g) \quad \mbox{and} \quad \operatorname{Aut}(\mathfrak g;\Gamma_G)  = \operatorname{Aut}(\mathfrak g).
  \]
  Thus the simply connected case is recovered as the particular case in which no global obstruction appears.

  For non-simply connected Lie groups, the subgroup $\Gamma_G$ restricts both the derivations that generate linear flows on $G$ and the automorphisms that  may be used to conjugate such flows. Therefore, the global topology of the Lie  group enters the algebraic classification through the pair
  \[
    \left(  \operatorname{Der}(\mathfrak g;\Gamma_G), \operatorname{Aut}(\mathfrak g;\Gamma_G)  \right).
  \]
\end{remark}

%---------------
%---------------

\section{Fixed Point Subgroups as Obstructions}

In this section we record some consequences of the fixed point subgroup of a linear flow. The purpose is to extract obstructions to algebraic conjugacy. There are two different levels of obstruction. The first one is infinitesimal and is detected by the kernel of the associated derivation. The second one is global and appears, for quotient Lie groups, through the requirement that the discrete central subgroup defining the quotient must be fixed by the lifted flow.

Let $(\varphi_t)_{t\in\mathbb R}$ be a linear flow on a Lie group $G$. We denote
its fixed point set by
\[
  \operatorname{Fix}(\varphi)
  =
  \{g\in G;\ \varphi_t(g)=g,\ \forall t\in\mathbb R\}.
\]

\begin{lemma}
  Let $(\varphi_t)_{t\in\mathbb R}$ be a linear flow on a Lie group $G$. Then $\operatorname{Fix}(\varphi)$ is a closed Lie subgroup of $G$.
\end{lemma}
\begin{proof}
  For each $t\in\mathbb R$, the set $\operatorname{Fix}(\varphi_t) =  \{g\in G;\ \varphi_t(g)=g\}$ is closed, since $\varphi_t$ is continuous. Hence
  \[
    \operatorname{Fix}(\varphi) = \bigcap_{t\in\mathbb R}\operatorname{Fix}(\varphi_t)
  \]
  is closed. Moreover, if $g,h\in\operatorname{Fix}(\varphi)$, then, since each $\varphi_t$ is an automorphism,
  \[
    \varphi_t(gh^{-1}) = \varphi_t(g)\varphi_t(h)^{-1} = gh^{-1}
  \]
  for every $t\in\mathbb R$. Thus $gh^{-1}\in\operatorname{Fix}(\varphi)$. Therefore $\operatorname{Fix}(\varphi)$ is a closed subgroup of $G$. By the closed subgroup theorem, it is a Lie subgroup.
\end{proof}

The fixed point subgroup is preserved by algebraic conjugacy.

\begin{proposition}
  Let $(\varphi_t)_{t\in\mathbb R}$ and $(\psi_t)_{t\in\mathbb R}$ be linear  flows on Lie groups $G$ and $H$, respectively. Suppose that there exists a  Lie group isomorphism $\Phi:G\longrightarrow H$ such that
  \[
    \Phi\circ\varphi_t=\psi_t\circ\Phi
  \]
  for every $t\in\mathbb R$. Then
  \[
    \Phi(\operatorname{Fix}(\varphi)) = \operatorname{Fix}(\psi).
  \]
  In particular, if $\operatorname{Fix}(\varphi)$ and $\operatorname{Fix}(\psi)$ are not isomorphic as Lie groups, then the flows are not algebraically conjugate.
\end{proposition}
\begin{proof}
  If $g\in\operatorname{Fix}(\varphi)$, then
  \[
    \psi_t(\Phi(g)) = \Phi(\varphi_t(g))  = \Phi(g)
  \]
  for every $t\in\mathbb R$. Hence  $\Phi(g)\in\operatorname{Fix}(\psi)$. This gives
  \[
    \Phi(\operatorname{Fix}(\varphi)) \subset \operatorname{Fix}(\psi).
  \]
  Applying the same argument to $\Phi^{-1}$, we obtain the reverse inclusion.
\end{proof}

We now relate this obstruction to the derivation associated with a linear flow.

\begin{proposition}
  Let $G$ be a connected Lie group and let $(\varphi_t)_{t\in\mathbb R}$ be a linear flow on $G$. If $D\in\operatorname{Der}(\mathfrak g)$ is its associated  derivation, then
  \[
    \operatorname{Lie}(\operatorname{Fix}(\varphi))=\ker D.
  \]
\end{proposition}
\begin{proof}
  By the previous lemma, $\operatorname{Fix}(\varphi)$ is a Lie subgroup of $G$,  so its Lie algebra is well defined. Let $X\in\ker D$. Then $e^{tD}X=X$ for every $t\in\mathbb R$. Since   $(d\varphi_t)_e=e^{tD}$ and $\varphi_t$ is an automorphism, we have
  \[
    \varphi_t(\exp(sX)) = \exp(se^{tD}X)  = \exp(sX)
  \]
  for all $s,t\in\mathbb R$. Hence the one-parameter subgroup $s\mapsto\exp(sX)$ is contained in $\operatorname{Fix}(\varphi)$, and so
  \[
    X\in \operatorname{Lie}(\operatorname{Fix}(\varphi)).
  \]

  Conversely, let $X\in \operatorname{Lie}(\operatorname{Fix}(\varphi))$. Then  $\exp(sX)\in\operatorname{Fix}(\varphi)$ for $s$ sufficiently small. Thus
  \[
    \varphi_t(\exp(sX))=\exp(sX).
  \]
  Using again that $\varphi_t$ is an automorphism, this becomes $\exp(se^{tD}X)=\exp(sX)$. For $s$ sufficiently small, differentiating at $s=0$ gives
  \[
    e^{tD}X=X
  \]
  for every $t\in\mathbb R$. Differentiating at $t=0$, we obtain $DX=0$. Thus $X\in\ker D$. The equality follows.
\end{proof}

As a consequence, the kernel of the associated derivation is an infinitesimal
invariant of algebraic conjugacy.

\begin{corollary}
  Let $(\varphi_t)_{t\in\mathbb R}$ and $(\psi_t)_{t\in\mathbb R}$ be algebraically conjugate linear flows on connected Lie groups $G$ and $H$. Let $D\in\operatorname{Der}(\mathfrak g)$ and $F\in\operatorname{Der}(\mathfrak h)$ be their associated derivations. If $\Phi:G\longrightarrow H$ is an algebraic conjugacy and $A=d\Phi_e:\mathfrak g\longrightarrow\mathfrak h$, then
  \[
    A(\ker D)=\ker F.
  \]
  In particular, if $\ker D$ and $\ker F$ are not isomorphic as Lie algebras, then the corresponding linear flows cannot be algebraically conjugate.
\end{corollary}
\begin{proof}
  By the invariance of the fixed point subgroup,
  \[
    \Phi(\operatorname{Fix}(\varphi)) = \operatorname{Fix}(\psi).
  \]
  Taking Lie algebras and using
  \[
    \operatorname{Lie}(\operatorname{Fix}(\varphi))=\ker D, \qquad  \operatorname{Lie}(\operatorname{Fix}(\psi))=\ker F,
  \]
  we obtain
  \[
    A(\ker D)=\ker F.
  \]
\end{proof}

This gives a first obstruction to algebraic conjugacy. If $D\in\operatorname{Der}(\mathfrak g)$, then the lifted linear flow$(\widetilde\varphi_t^D)_{t\in\mathbb R}$ descends to $G$ if and only if $\Gamma_G\subset \operatorname{Fix}(\widetilde\varphi^D)$.
Although
\[
  \operatorname{Lie}(\operatorname{Fix}(\widetilde\varphi^D))=\ker D,
\]
the condition $\Gamma_G\subset \operatorname{Fix}(\widetilde\varphi^D)$ is not determined only by $\ker D$. Indeed, $\Gamma_G$ is discrete, and therefore $\operatorname{Lie}(\Gamma_G)=0$. Thus the projectability condition depends on specific global fixed points of the  lifted flow, not merely on the infinitesimal fixed point algebra.

Consequently, two distinct obstructions must be kept separate:
\[
  A(\ker D)=\ker F
\]
is an infinitesimal obstruction to conjugacy, whereas $\Gamma_G\subset \operatorname{Fix}(\widetilde\varphi^D)$ is a global obstruction to the existence of the projected linear flow on the  quotient.

\begin{example}
  Let
  \[
    \widetilde G=\mathbb R^2, \qquad  G=\mathbb R\times S^1\simeq \mathbb R^2/\Gamma_G, \qquad  \Gamma_G=\{(0,n);\ n\in\mathbb Z\}.
  \]
  As in Example \ref{ex:projectable_derivations_cylinder}, the projectable derivations are precisely the matrices
  \[
    D=
    \begin{pmatrix}
      a & 0\\
      c & 0
    \end{pmatrix}.
  \]
  Consider
  \[
    D=
    \begin{pmatrix}
      1 & 0\\
      0 & 0
    \end{pmatrix},
    \qquad
    F=
    \begin{pmatrix}
      0 & 0\\
      0 & 0
    \end{pmatrix}.
  \]
  Then
  \[
    \ker D=\operatorname{span}\{(0,1)\},  \qquad  \ker F=\mathbb R^2.
  \]
  Hence
  \[
    \dim\ker D=1, \qquad  \dim\ker F=2.
  \]
  Therefore $\ker D$ and $\ker F$ are not isomorphic as vector spaces, and the linear flows generated by $D$ and $F$ on $G$ cannot be algebraically conjugate.
\end{example}

%-------
%-------

\section{Linear Flows on the Heisenberg Group}

In this section we illustrate the preceding results in the three-dimensional Heisenberg Lie algebra. The purpose is twofold. First, we describe explicitly the action of the automorphism group on the derivation algebra. Second, we show how the projectability condition restricts the admissible derivations when one passes from the simply connected Heisenberg group to a quotient by a discrete central
subgroup.

Let $\mathfrak h_3$ be the Lie algebra with basis $\{X,Y,Z\}$ and Lie bracket determined by
\[
[X,Y]=Z,\qquad [X,Z]=[Y,Z]=0.
\]
Thus $Z(\mathfrak h_3)=[\mathfrak h_3,\mathfrak h_3]=\operatorname{span}\{Z\}$. 

We first recall the explicit form of the derivations of $\mathfrak h_3$. Every derivation $D\in\operatorname{Der}(\mathfrak h_3)$ is represented, with respect to the basis $\{X,Y,Z\}$, by a matrix of the form
\[
  D=
  \begin{pmatrix}
    a & b & 0\\
    d & e & 0\\
    c & f & a+e
  \end{pmatrix},
  \qquad a,b,c,d,e,f\in\mathbb R.
  \]
Conversely, every matrix of this form defines a derivation of $\mathfrak h_3$.  In particular, $\dim\operatorname{Der}(\mathfrak h_3)=6$.

We shall write such a derivation in block form as
\[
  D=
  \begin{pmatrix}
  A & 0\\
  w & \tau
  \end{pmatrix},
\]
where
\[
  A=
  \begin{pmatrix}
  a & b\\
  d & e
  \end{pmatrix},
  \qquad
  w=\begin{pmatrix}c & f\end{pmatrix},
  \qquad
  \tau=\operatorname{tr}(A)=a+e.
\]
If $\lambda_1,\lambda_2$ are the eigenvalues of $A$, then the eigenvalues of  $D$ are $\lambda_1$, $\lambda_2$, $\lambda_1+\lambda_2$.

We also recall the automorphism group of $\mathfrak h_3$. Every automorphism $T\in\operatorname{Aut}(\mathfrak h_3)$ is represented, with respect to the basis $\{X,Y,Z\}$, by a matrix of the form
\[
  T=
  \begin{pmatrix}
    p & q & 0\\
    s & t & 0\\
    r & u & pt-qs
  \end{pmatrix},
  \qquad pt-qs\neq 0.
\]
Equivalently,
\[
  T=
  \begin{pmatrix}
    P & 0\\
    v & \det P
  \end{pmatrix},
  \qquad
  P\in GL(2,\mathbb R),\quad v\in\mathbb R^{1\times 2}.
\]

By the simply connected case of the conjugacy criterion, the algebraic  classification of linear flows on the connected and simply connected Heisenberg group is equivalent to the classification of the orbits of the action
\[
  \operatorname{Aut}(\mathfrak h_3)\times \operatorname{Der}(\mathfrak h_3)
  \longrightarrow
  \operatorname{Der}(\mathfrak h_3), \quad (T,D)\longmapsto TDT^{-1}.
\]

Let
\[
  D=
  \begin{pmatrix}
    A & 0\\
    w & \tau
  \end{pmatrix}
  \in\operatorname{Der}(\mathfrak h_3), \quad \tau=\operatorname{tr}(A),
\]
and let
\[
  T=
  \begin{pmatrix}
    P & 0\\
    v & \det P
  \end{pmatrix}
  \in\operatorname{Aut}(\mathfrak h_3).
\]
A direct computation gives
\[
  TDT^{-1}
  =
  \begin{pmatrix}
    PAP^{-1} & 0\\
    w' & \tau
  \end{pmatrix},
\]
where
\[
  w' = \big((\det P)w+v(A-\tau I)\big)P^{-1}.
\]
Therefore the horizontal block transforms by similarity,
\[
  A\longmapsto PAP^{-1},
\]
whereas the central row transforms according to
\[
  w\longmapsto  \big((\det P)w+v(A-\tau I)\big)P^{-1}.
\]
The following consequence is useful for reducing the derivation.

\begin{proposition}\label{prop:eliminate_w_h3}
  Let
  \[
    D=
    \begin{pmatrix}
      A & 0\\
      w & \tau
    \end{pmatrix}
    \in\operatorname{Der}(\mathfrak h_3), \quad \tau=\operatorname{tr}(A).
  \]
  If $\det(A-\tau I)\neq 0$, then $D$ is conjugate, by an automorphism of $\mathfrak h_3$, to a derivation of the form
  \[
    \widetilde D=
    \begin{pmatrix}
      A & 0\\
      0 & \tau
  \end{pmatrix}.
  \]
\end{proposition}
\begin{proof}
Taking $P=I$, the transformation formula for $w$ becomes
\[
w\longmapsto w+v(A-\tau I).
\]
If $A-\tau I$ is invertible, choose
\[
v=-w(A-\tau I)^{-1}.
\]
Then the transformed central row is zero.
\end{proof}

Since $\tau=\operatorname{tr}(A)$, for a $2\times 2$ matrix $A$ one has
\[
\det(A-\tau I)=\det A.
\]
Thus the obstruction to eliminating the central row $w$ occurs precisely when
the horizontal block $A$ is singular.

We now record reduced representatives for the action of $\operatorname{Aut}(\mathfrak h_3)$ on $\operatorname{Der}(\mathfrak h_3)$. These representatives should be understood as reduced forms; in the case $\det A\neq 0$, the remaining classification is the usual similarity classification of the $2\times 2$ matrix $A$.

\begin{theorem}\label{thm:reduced_forms_h3}
  Every derivation of $\mathfrak h_3$ is conjugate, by an automorphism of $\mathfrak h_3$, to one of the following reduced forms.
  \begin{enumerate}
    \item If $\det A\neq 0$, then
    \[
      D\sim
      \begin{pmatrix}
        A & 0\\
        0 & \operatorname{tr}(A)
      \end{pmatrix},
    \]
    where $A$ is taken up to similarity in $GL(2,\mathbb R)$.
    \item If $\det A=0$, $A\neq 0$, and $\tau=\operatorname{tr}(A)\neq 0$, then
    \[
      D\sim D_{\tau,\varepsilon}  =
      \begin{pmatrix}
        \tau & 0 & 0\\
        0 & 0 & 0\\
        \varepsilon & 0 & \tau
      \end{pmatrix}, \quad \varepsilon\in\{0,1\}.
    \]
    \item If $\det A=0$, $A\neq 0$, and $\tau=\operatorname{tr}(A)=0$, then
    \[
      D\sim D_{\mathrm{nil},\varepsilon}  
      =
      \begin{pmatrix}
        0 & 1 & 0\\
        0 & 0 & 0\\
        \varepsilon & 0 & 0
      \end{pmatrix}, \quad \varepsilon\in\{0,1\}.
    \]
    \item If $A=0$, then
    \[
      D\sim 0 \quad \text{or} 
      \quad 
      D\sim D_{\mathrm{cent}}
      =
      \begin{pmatrix}
        0 & 0 & 0\\
        0 & 0 & 0\\
        1 & 0 & 0
      \end{pmatrix}.
    \]
  \end{enumerate}
\end{theorem}
 \begin{proof}
  If $\det A\neq 0$, then $A-\tau I$ is invertible. By  Proposition \ref{prop:eliminate_w_h3}, the central row $w$ can be eliminated. Thus the classification reduces to the similarity classification of the horizontal  block $A$.

  Assume now that $\det A=0$, $A\neq 0$, and $\tau=\operatorname{tr}(A)\neq 0$. Then $A$ has eigenvalues $\tau$ and $0$, and is similar to
  \[
    \begin{pmatrix}
      \tau & 0\\
      0 & 0
    \end{pmatrix}.
  \]
  In this basis,
  \[
    A-\tau I=
    \begin{pmatrix}
      0 & 0\\
      0 & -\tau
    \end{pmatrix}.
  \]
  The term $v(A-\tau I)$ allows one component of $w$ to be eliminated. The  remaining component is either zero or nonzero. Automorphisms preserving the diagonal form of $A$ allow any nonzero remaining component to be normalized to  $1$. This gives the forms $D_{\tau,\varepsilon}$.

  If $\det A=0$, $A\neq 0$, and $\tau=0$, then $A$ is a nonzero nilpotent matrix. Hence it is similar to
  \[
    \begin{pmatrix}
      0 & 1\\
      0 & 0
    \end{pmatrix}.
  \]
  The transformation formula for $w$ again allows one component to be eliminated, and the remaining component is either zero or can be normalized to $1$. This  gives the forms $D_{\mathrm{nil},\varepsilon}$.

  Finally, if $A=0$, then $\tau=0$ and
  \[
    D=
    \begin{pmatrix}
      0 & 0\\
      w & 0
    \end{pmatrix}.
  \]
  The action on $w$ reduces to $w\longmapsto (\det P)wP^{-1}$. There are exactly two possibilities: $w=0$ and $w\neq 0$. These correspond, respectively, to $D=0$ and to $D_{\mathrm{cent}}$.
\end{proof}

Thus, on the connected and simply connected Heisenberg group, every derivation of $\mathfrak h_3$ generates a linear flow, and the algebraic conjugacy classes of such flows are determined by the orbits of $\operatorname{Aut}(\mathfrak h_3)$
on $\operatorname{Der}(\mathfrak h_3)$.

\subsection{Projectable derivations on central quotients}

We now show how the classification changes when one passes to a non-simply connected quotient.

Let $\widetilde H_3$ be the connected and simply connected Heisenberg group with Lie algebra $\mathfrak h_3$. Since $\widetilde H_3$ is nilpotent and simply connected, the exponential map is a diffeomorphism. The center of
$\widetilde H_3$ is $Z(\widetilde H_3)=\exp(\mathbb R Z)$. For $a>0$, consider the discrete central subgroup
\[
  \Gamma_a=\{\exp(naZ);\, n\in\mathbb Z\}.
\]
Set $H_{3,a}:=\widetilde H_3/\Gamma_a$.

\begin{proposition}\label{prop:projectable_derivations_h3_quotient}
  Let
  \[
    D=
    \begin{pmatrix}
      A & 0\\
      w & \tau
    \end{pmatrix}
    \in\operatorname{Der}(\mathfrak h_3), \quad \tau=\operatorname{tr}(A).
  \]
  Then $D$ defines a linear flow on $H_{3,a}=\widetilde H_3/\Gamma_a$ if and only if $\tau=0$. Equivalently,
  \[
    \operatorname{Der}(\mathfrak h_3;\Gamma_a)  =
    \left\{
    \begin{pmatrix}
      A & 0\\
      w & 0
    \end{pmatrix}
    \in\operatorname{Der}(\mathfrak h_3); \operatorname{tr}(A)=0  \right\}.
  \]
\end{proposition}
\begin{proof}
  Let $(\widetilde\varphi_t^D)_{t\in\mathbb R}$ be the linear flow on $\widetilde H_3$ generated by $D$. Since $\widetilde H_3$ is simply connected and nilpotent,
  \[
    \widetilde\varphi_t^D(\exp U)=\exp(e^{tD}U),  \qquad U\in\mathfrak h_3.
  \]
  Moreover, $D(Z)=\tau Z$. Hence $e^{tD}Z=e^{t\tau}Z$. For the generator $\exp(aZ)$ of $\Gamma_a$, we obtain
  \[
    \widetilde\varphi_t^D(\exp(aZ)) = \exp(ae^{t\tau}Z).
  \]
  The flow descends to $H_{3,a}$ if and only if $\Gamma_a\subset \operatorname{Fix}(\widetilde\varphi^D)$. This is equivalent to
  \[
    \exp(ae^{t\tau}Z)=\exp(aZ)
  \]
  for every $t\in\mathbb R$. Since the exponential map is injective on  $\mathbb R Z$, this is equivalent to $ae^{t\tau}=a$ for every $t\in\mathbb R$. Since $a>0$, this holds if and only if $\tau=0$. Therefore $D$ is projectable precisely when $\operatorname{tr}(A)=0$.
\end{proof}

This example shows explicitly that
\[
  \operatorname{Der}(\mathfrak h_3;\Gamma_a)  \subsetneq  \operatorname{Der}(\mathfrak h_3).
\]
Thus, although every derivation of $\mathfrak h_3$ generates a linear flow on the simply connected group $\widetilde H_3$, only the derivations satisfying $\operatorname{tr}(A)=0$ generate linear flows on the quotient $H_{3,a}$.

The admissible automorphisms are also restricted.

\begin{proposition}\label{prop:admissible_automorphisms_h3_quotient}
  Let
  \[
    T=
    \begin{pmatrix}
      P & 0\\
      v & \det P
    \end{pmatrix}
    \in\operatorname{Aut}(\mathfrak h_3).
  \]
  Then $T$ is admissible for the quotient $H_{3,a}$, that is, $\widetilde T(\Gamma_a)=\Gamma_a$, if and only if $\det P=\pm 1$. Consequently,
  \[
    \operatorname{Aut}(\mathfrak h_3;\Gamma_a)
    =
    \left\{
    \begin{pmatrix}
      P & 0\\
      v & \det P
    \end{pmatrix}
    \in\operatorname{Aut}(\mathfrak h_3);
    \det P=\pm 1
    \right\}.
  \]
\end{proposition}
\begin{proof}
  The automorphism $T$ satisfies $T(Z)=(\det P)Z$. Therefore the corresponding automorphism $\widetilde T$ of $\widetilde H_3$ satisfies
  \[
    \widetilde T(\exp(aZ))=\exp(a(\det P)Z).
  \]
  Thus $\widetilde T(\Gamma_a)=\Gamma_a$ if and only if multiplication by $\det P$ maps the lattice $a\mathbb Z$ onto  itself. This happens if and only if $\det P=\pm 1$.
\end{proof}

Therefore, on the quotient $H_{3,a}$, the algebraic classification of linear flows is governed by the restricted action
\[
  \operatorname{Aut}(\mathfrak h_3;\Gamma_a)  \times  \operatorname{Der}(\mathfrak h_3;\Gamma_a)  \longrightarrow \operatorname{Der}(\mathfrak h_3;\Gamma_a),
\]
and not by the full action of $\operatorname{Aut}(\mathfrak h_3)$ on $\operatorname{Der}(\mathfrak h_3)$.

We now describe the effect of the projectability condition on the reduced representatives. For a derivation $D \in \operatorname{Der}(\mathfrak h_3;\Gamma_a)$ and a a isomorphism $T \in \operatorname{Aut}(\mathfrak h_3;\Gamma_a)$ we see that 
\[
  TDT^{-1}  =
  \begin{pmatrix}
    PAP^{-1} & 0\\
    w' & 0
  \end{pmatrix},
  \]
where $w' = \big((\det P)w+vA\big)P^{-1}$ with $\det P=\pm 1$. 

The next result shows that the quotient classification is not obtained merely by imposing $\operatorname{tr}(A)=0$ on the simply connected classification. Since the admissible automorphisms satisfy $\det P=\pm1$, some normalizations available in the simply connected case are no longer possible. In particular, the nilpotent case contains a continuous invariant.

\begin{theorem}\label{thm:reduced_forms_h3_quotient}
  Every projectable derivation $D\in\operatorname{Der}(\mathfrak h_3;\Gamma_a)$ is conjugate, by an admissible automorphism in $\operatorname{Aut}(\mathfrak h_3;\Gamma_a)$, to one of the following reduced forms.
  \begin{enumerate}
    \item If $A$ is invertible and $\det A<0$, then
    \[
      D\sim D_{\lambda}^{\mathrm{hyp}}
      =
      \begin{pmatrix}
        \lambda & 0 & 0\\
        0 & -\lambda & 0\\
        0 & 0 & 0
      \end{pmatrix},
      \qquad \lambda>0.
    \]
    \item If $A$ is invertible and $\det A>0$, then
    \[
      D\sim D_{\beta}^{\mathrm{ell}}
      =
      \begin{pmatrix}
        0 & -\beta & 0\\
        \beta & 0 & 0\\
        0 & 0 & 0
     \end{pmatrix},
     \qquad \beta>0.
    \]
    \item If $A\neq 0$ and $A$ is nilpotent, then
    \[
      D\sim D_{\mathrm{nil},\rho}
      =
      \begin{pmatrix}
        0 & 1 & 0\\
        0 & 0 & 0\\
        \rho & 0 & 0
      \end{pmatrix},
      \qquad \rho\geq 0.
    \]
    \item If $A=0$, then
    \[
      D\sim 0 \quad \text{or} \quad D\sim D_{\mathrm{cent}}
      =
      \begin{pmatrix}
        0 & 0 & 0\\
        0 & 0 & 0\\
        1 & 0 & 0
      \end{pmatrix}.
    \]
  \end{enumerate}
\end{theorem}
\begin{proof}
  Let
  \[
    D=
    \begin{pmatrix}
      A & 0\\
      w & 0
    \end{pmatrix}
    \in\operatorname{Der}(\mathfrak h_3;\Gamma_a),  \qquad A\in\mathfrak{sl}(2,\mathbb R).
  \]
  Assume first that $A$ is invertible. Since
  \[
    w'  = \big((\det P)w+vA\big)P^{-1},
  \]
  we may take $P=I$ and choose $v=-wA^{-1}$.Then the transformed central row is zero. Hence $D$ is conjugate to
  \[
    \begin{pmatrix}
      A & 0\\
      0 & 0
    \end{pmatrix}.
  \]
  Since $A\in\mathfrak{sl}(2,\mathbb R)$, its characteristic polynomial is $\lambda^2+\det A$. If $\det A<0$, then $A$ has real eigenvalues $\lambda$ and $-\lambda$, with $\lambda=\sqrt{-\det A}>0$, and $A$ is similar to
  \[
    \begin{pmatrix}
      \lambda & 0\\
      0 & -\lambda
    \end{pmatrix}.
  \]
  If $\det A>0$, then $A$ has purely imaginary eigenvalues $\pm i\beta$, $\beta=\sqrt{\det A}>0$, and $A$ is similar over $\mathbb R$ to
  \[
    \begin{pmatrix}
      0 & -\beta\\
      \beta & 0
    \end{pmatrix}.
  \]
  The similarity may be realized by a matrix $P$ with $\det P=\pm 1$, since multiplying a conjugating matrix by a nonzero scalar does not change the  conjugation and allows its determinant to be normalized to $\pm 1$. This gives  the first two reduced forms.

  Assume now that $A\neq 0$ is nilpotent. Again, after conjugation by an admissible matrix $P$, we may suppose that
  \[
    A=  N:=
    \begin{pmatrix}
      0 & 1\\
      0 & 0
    \end{pmatrix}.
  \]
  Write $w=(\alpha,\beta)$. Since $vN=(0,v_1)$ for $v=(v_1,v_2)$, the term $vN$ allows us to eliminate the second component of $w$. Thus $D$ is conjugate to a derivation with $A=N$, $w=(\alpha,0)$. The stabilizer of $N$ inside the admissible group consists of matrices
  \[
    P=
    \begin{pmatrix}
      \varepsilon & q\\
      0 & \varepsilon
    \end{pmatrix},
    \qquad
    \varepsilon=\pm 1,\quad q\in\mathbb R.
  \]
  Using again the freedom given by $vN$, this stabilizer changes $\alpha$ only  by sign. Therefore the remaining invariant is $    \rho=|\alpha|\geq 0$. This gives the family
  \[
    D_{\mathrm{nil},\rho} =
    \begin{pmatrix}
      0 & 1 & 0\\
      0 & 0 & 0\\
      \rho & 0 & 0
    \end{pmatrix},
    \qquad \rho\geq 0.
  \]

  Finally, suppose that $A=0$. Then
  \[
    D=
    \begin{pmatrix}
      0 & 0\\
      w & 0
    \end{pmatrix}.
  \]
  The conjugation formula reduces to $w\longmapsto (\det P)wP^{-1}$, $\det P=\pm 1$. If $w=0$, then $D=0$. If $w\neq 0$, one can choose $P$ with $\det P=\pm 1$ such that
  \[
    (\det P)wP^{-1}=(1,0).
  \]
  Hence every nonzero central row is conjugate to
  \[
    D_{\mathrm{cent}} =
    \begin{pmatrix}
      0 & 0 & 0\\
      0 & 0 & 0\\
      1 & 0 & 0
    \end{pmatrix}.
  \]
  The result follows.
\end{proof}

Theorem \ref{thm:reduced_forms_h3_quotient} shows that the quotient case is not obtained merely by repeating the simply connected classification. The projectability condition first imposes
\[
  \operatorname{tr}(A)=0,
\]
and the admissibility condition restricts the conjugating automorphisms to those  with $\det P=\pm 1$. Consequently, the algebraic classification on $H_{3,a}$ is governed by the restricted action
\[
  \operatorname{Aut}(\mathfrak h_3;\Gamma_a)  \curvearrowright  \operatorname{Der}(\mathfrak h_3;\Gamma_a),
\]
rather than by the full action of
\[
  \operatorname{Aut}(\mathfrak h_3) \curvearrowright  \operatorname{Der}(\mathfrak h_3).
\]

%================
%================

\end{document}